\documentclass[10pt,a4paper]{article}     
\usepackage{setspace,graphicx,amsmath,amssymb,graphics,float,fullpage,multirow,fancyhdr,amsthm,color,url,enumerate,subfig,pgf,tikz,mathrsfs,longtable} 
\usepackage[font=small,labelfont=bf]{caption}
\usepackage[toc,page]{appendix}

\allowdisplaybreaks  
\usetikzlibrary{arrows}
\def\Pb{\mathbb{P}}
\def\Eb{\mathbb{E}}
\def\Rb{\mathbb{R}}

\def\Nb{\mathbb{N}}
\def\Zb{\mathbb{Z}}

\def\Bb{\mathbb{B}}
\def\Tc{\mathcal{T}}

\def\Cc{\mathcal{C}}

\def\Ac{\mathcal{A}}

\def\Cc{\mathcal{C}}

\def\Hc{\mathcal{H}}
\def\Bc{\mathcal{B}}

\def\Pr{\mathbf{P}}
\def\Er{\mathbf{E}}

\def\Ts{\mathscr{T}}
\def\Pt{\mathit{P}}
\def\Et{\mathit{E}}
\def\th{^{\text{th}}}
\def\oT{{\overline{\Tc}}}
\def\oR{{\overline{\rho}}}

\def\|{\\ \smallskip}

\def\ind{\mathbf{1}}
\def\nin{n \rightarrow \infty}

\def\ed{\stackrel{\text{\tiny{d}}}{=}}
\def\qqquad{\qquad \quad}
\def\qqqquad{\qqquad \quad}

\usetikzlibrary{arrows}

\newtheorem{thm}{Theorem}
\newtheorem{lem}{Lemma}[section]
\newtheorem{prp}[lem]{Proposition}
\newtheorem{cly}[lem]{Corollary}

\newtheorem{rem}[lem]{Remark}

\numberwithin{equation}{section}
\title{A quenched central limit theorem for biased random walks on supercritical Galton-Watson trees}
\author{Adam Bowditch, University of Warwick}
\date{}
\begin{document}             
\maketitle
\begin{abstract}
In this note, we prove a quenched functional central limit theorem for a biased random walk on a supercritical Galton-Watson tree with leaves. 
This extends a result of Peres and Zeitouni (2008) where the case without leaves is considered. 
A conjecture of Ben Arous and Fribergh (2016) suggests an upper bound on the bias which we observe to be sharp.
\end{abstract}

\let\thefootnote\relax\footnote{\textit{MSC2010 subject classifications:} Primary 60K37, 60F05, 60F17; secondary 60J80. }

\section{Introduction}
We investigate biased random walks on supercritical Galton-Watson trees with leaves. A GW-tree conditioned to survive consists of an infinite backbone structure with finite trees attached as branches. The backbone structure is a GW-tree whose offspring law does not have deaths. The branches are formed by attaching a random number of independent GW-trees (conditioned to die out) to each vertex of the backbone. This forms dead-ends in the environment which makes it a natural setting for observing trapping as the walk is slowed by taking excursions in the branches of the tree. The number of trees attached to a given vertex on the backbone has a distribution depending on the backbone locally. This means that the branches are not i.i.d.\ and therefore the trapping incurred by the addition of the leaves demonstrates a significant complication to the model without leaves studied in \cite{peze}.

The influence of the bias on the trapping is an important feature of the model. As the bias is increased, the local drift away from the root will increase but this does not necessarily speed up the walk. This is because it increases the time trapped in the finite leaves from which the walk cannot escape without taking long sequences of movements against the bias. In \cite{lypepe} it is shown that, for a suitably large bias, the trapping is sufficient to slow the walk to zero speed whereas, for small bias, the expected trapping time is finite and the walk has a positive limiting speed. Under a further restriction on the bias the trapping times have finite variance; we use this to prove a quenched invariance principle for the walk.  

We next describe the supercritical GW-tree in greater detail. Let $\{p_k\}$ denote the offspring distribution of a GW-process $Z_n$ with a single progenitor, mean $\mu >1$ and probability generating function $f$. The process $Z_n$ gives rise to a random tree $\Tc^f$ where individuals are represented by vertices and edges connect individuals with their offspring. We denote by $\rho$ the root, which is the vertex representing the unique progenitor. Let $q$ denote the extinction probability of $Z_n$ which is non-zero only when $p_0>0$. In this case we then define 
\[g(s):=\frac{f(s)-f(qs)}{1-q} \qquad \text{ and } \qquad h(s):=\frac{f(qs)}{q}\]
which are generating functions of GW-processes. In particular, $g$ is the generating function of a GW-process without deaths and $h$ is the generating function of a subcritical GW-process. An $f$-GW-tree conditioned on nonextinction $\Tc$ can be generated by first generating a $g$-GW-tree $\Tc^g$ and then, to each vertex $x$ of $\Tc^g$, appending a random number $M_x$ of independent $h$-GW-trees. We refer to $\Tc^g$ as the backbone of $\Tc$ and the finite trees appended to $\Tc^g$ as the traps. 

We now introduce the biased random walk on a fixed tree. For $\Tc$ fixed with $x\in \Tc$ let $\overleftarrow{x}$ denote the parent of $x$ and $c(x)$ the set of children of $x$. A $\beta$-biased random walk on $\Tc$ is a random walk $(X_n)_{n \geq 0}$ on the vertices of $\Tc$ which is $\beta$-times more likely to make a transition to a given child of the current vertex than the parent (which are the only options). More specifically, the random walk started from $z \in \Tc$ is the Markov chain defined by $\Pt^\Tc_z(X_0=z)=1$ and the transition probabilities 
\[\Pt^\Tc_z(X_{n+1}=y|X_n=x)=\begin{cases} \frac{1}{1+\beta |c(x)|}, & \text{if } y=\overleftarrow{x}, \\  \frac{\beta}{1+\beta |c(x)|}, & \text{if } y \in c(x), \; x \neq \rho, \\ \frac{1}{
d_\rho}, & \text{if } y \in c(x), \; x =\rho, \\ 0, & \text{otherwise.} \\ \end{cases} \]  
We use $\Pb_\rho(\cdot):=\int \Pt^\Tc_\rho(\cdot)\Pr(\text{d}\Tc)$ for the annealed law obtained by averaging the quenched law $\Pt^\Tc_\rho$ over the law $\Pr$ on $f$-GW-trees conditioned to survive. Unless indicated otherwise, we start the walk at $\rho$.

For $x \in \Tc$, let $|x|:=d(\rho,x)$ denote the distance between $x$ and the root of the tree. It has been shown in \cite{lypepe} that if $\beta\in(\mu^{-1},f'(q)^{-1})$ then $|X_n|n^{-1}$ converges $\Pb$-a.s.\ to a deterministic constant $\nu>0$ called the speed of the walk. When $\beta<\mu^{-1}$ the walk is recurrent and $|X_n|n^{-1}$ converges $\Pb$-a.s.\ to $0$. When the bias is large the walk is transient but slowed by having to make long sequences of movements against the bias in order to escape the traps; in particular, if $\beta\geq f'(q)^{-1}$ then the slowing affect is strong enough to cause $|X_n|n^{-1}$ to converge $\Pb$-a.s.\ to $0$. This regime has been studied further by \cite{arfrgaha} and \cite{bowd} where polynomial scaling results are shown. 

For $\sigma,t>0$ and $n=1,2,...$ define
\[B_t^n:=\frac{|X_{\lfloor nt\rfloor}|-n \nu t}{\sigma\sqrt{n}}.\]
Our main result, Theorem \ref{t:SupQCLT}, is a quenched invariance principle for $B_t^n$.
\begin{thm}\label{t:SupQCLT}
 Suppose $p_0>0$, $\mu>1$, $\beta \in(\mu^{-1},f'(q)^{-1/2})$ and that there exists $\lambda>1$ such that
\begin{flalign}\label{e:expmom}
\sum_{k\geq 0}\lambda^kp_k<\infty.
\end{flalign}
There exists $\sigma>0$ such that, for $\Pr$-a.e.\ $\Tc$, we have that the process $(B_t^n)_{t\geq 0}$ converges in $\Pt^\Tc$-distribution on $D([0,\infty),\Rb)$ endowed with the Skorohod $J_1$ topology to a standard Brownian motion.
 \end{thm}

The condition $p_0>0$ is so that the tree has leaves; the case without leaves, $p_0=0$, is considered in \cite{peze}. This no longer requires the condition $\beta<f'(q)^{-1/2}$ which is due to the trapping in the dead-ends caused by the leaves. Indeed, when $p_0=0$ we have that $q=0$ and $f'(0)=0$ therefore the condition becomes irrelevant. This regime is studied further in \cite{peze} to the case where $\beta=\mu^{-1}$ and $p_0=0$; in this setting $\nu=0$ and it is shown that $B_t^n$ converges in distribution to the absolute value of a Brownian motion. This result is extended in \cite{desu} to random walks on multi-type GW-trees with leaves. Although the dead-ends in this model trap the walk, the bias is small and therefore the slowing is weak. 

By choosing $p_0>0$ we allow the tree to have leaves; this creates traps in the environment which slow the walk. A key input for this work is the second moment estimate for trapping times in trees determined in \cite{bowdpre}. From this we infer that $\beta<f'(q)^{-1/2}$ is the correct upper bound on the bias so that the trapping times in our model have finite variance. We conclude, in Remark \ref{r:BetaMu}, that this upper bound is necessary which confirms \cite[Conjecture 3.1]{arfr}.

We assume that the exponential moments condition (\ref{e:expmom}) holds throughout. This is a purely technical assumption which we expect could be relaxed to a sufficiently large moment condition however the main focus of this note has been to obtain the optimal upper bound on the bias.

In \cite{bowdpre}, the second moment condition for trapping times in finite trees is used to prove an annealed invariance principle and quenched central limit theorem for a biased random walk on a subcritical GW-tree conditioned to survive. In that model the backbone is one-dimensional and the fluctuations are significantly influenced by the specific instance of the environment. This results in an environment dependent centring for the walk in the quenched central limit theorem. In the supercritical case the walk will randomly choose one of infinitely many escape routes; this creates a mixing of the environment which yields a deterministic centring in the quenched result Theorem \ref{t:SupQCLT}.

We begin, in Section \ref{s:SAnn}, by proving an annealed functional central limit theorem for the walk by adapting the renewal argument used in \cite[Theorem 4.1]{sz}. This is then extended to the quenched result Theorem \ref{t:SupQCLT} in Section \ref{s:SQue} by applying the argument used in \cite{bosz} which largely involves showing that multiple copies of the walk see sufficiently different areas of the tree.

\section{An annealed invariance principle}\label{s:SAnn}
We begin this section by showing that the time spent in a branch has finite variance. Let $\oT^h$ be the tree formed by attaching an additional vertex $\oR$ (as the parent of the root $\rho$) to an $h$-GW-tree $\Tc^h$. For a fixed tree $T$ and vertex $x \in T$ let $\tau_x^+:=\inf\{k>0:X_k=x\}$ denote the first return time to $x$. Let $\xi^f,\xi^g,\xi^h$ be random variables with probability generating functions $f,g$ and $h$ respectively then let $\xi$ be equal in distribution to the number of vertices in the first generation of $\Tc$. Since the generation sizes of $\Tc^g$ are dominated by those of $\Tc$ we have that $\xi^g$ is stochastically dominated by $\xi$. Using Bayes' law we have that $\Pr(\xi=k)= p_k(1-q^k)(1-q)^{-1} \leq cp_k$ therefore both $\xi$ and $\xi^g$ inherit the exponential moments of $\xi^f$. Furthermore $\Pr(\xi^h=k)= p_kq^k$ therefore $\xi^h$ automatically has exponential moments. 

\begin{lem}\label{l:subComp}
 Suppose that $p_0>0$, $\mu>1$ and $\beta \in(\mu^{-1},f'(q)^{-1/2})$, then we have that \[\Er\left[\Et^{\oT^h}_{\oR}\left[\left(\tau^+_{\oR}\right)^2\right]\right] <\infty.\]
 \begin{proof}
  We can write
  \[\tau^+_{\oR}=\sum_{x \in \oT^h}v_x \qquad \text{where} \qquad v_x=\sum_{k=1}^{\tau^+_{\oR}}\ind_{\{X_k=x\}}\]
  is the number of visits to $x$ before returning to $\oR$. Recall that $c(x)$ denotes the set of children of $x$. It then follows that
  \begin{flalign}
   \Er\left[\Et^{\oT^h}_{\oR}\left[\left(\tau^+_{\oR}\right)^2\right]\right] & = \Er\left[\sum_{x,y \in \oT^h}\Et^{\oT^h}_{\oR}[v_xv_y]\right] \notag \\
   & \leq C_\beta \Er\left[\sum_{x,y \in \oT^h}(|c(x)|\beta+1)(|c(y)|\beta+1)\beta^{|x|+|y|}\right] \notag \\
   & = C_\beta \Er\left[\sum_{x \in \oT^h}(|c(x)|\beta+1)\beta^{|x|}\sum_{y \in \oT^h}(|c(y)|\beta+1)\beta^{|y|}\right] \label{e:EtErTh}
  \end{flalign}
where the inequality follows from \cite[Lemma 4.8]{bowdpre}. Letting $Z_k^h$ denote the size of the $k\th$ generation of $\oT^h$ and collecting terms in each generation we have that 
\[\sum_{x \in \oT^h}(|c(x)|\beta+1)\beta^{|x|} \; = \; 1+\sum_{k\geq 1}Z_k^h(\beta^k+\beta^{k-1}) \; \leq \; (1+\beta^{-1})\sum_{k\geq 0}Z_k^h\beta^k.\]
By \cite[Lemma 4.1]{bowdpre}, since $\xi^h$ has exponential moments, we have that $\Er[Z_k^hZ_j^h]\leq Cf'(q)^j$ whenever $j\geq k$. Substituting this and the above inequality into (\ref{e:EtErTh}) we have that 
\begin{flalign*}
 \Er\left[\Et^{\oT^h}_{\oR}\left[\left(\tau^+_{\oR}\right)^2\right]\right] & \leq C_\beta \sum_{k\geq 0}\beta^k\sum_{j\geq k}\Er[Z_k^hZ_j^h]\beta^j \\
 & \leq C_\beta  \sum_{k\geq 0}\beta^k\sum_{j\geq k}(f'(q)\beta)^j \\
 & \leq C_{\beta,f'(q)}  \sum_{k\geq 0}(f'(q)\beta^2)^k 
\end{flalign*}
which is finite by the assumption that $\beta<f'(q)^{-1/2}$.
 \end{proof}
\end{lem}

Let $r(0):=0$, $r(n):=\inf\{k> r(n-1):X_k,X_{k-1} \in \Tc^g\}$ for $n\geq 1$ and $Y_n:=X_{r(n)}$, then $Y_n$ is a $\beta$-biased random walk on $\Tc^g$ coupled to $X_n$. Write $\zeta^Y_0:=0$ and for $m=1,2,...$ let 
\[\zeta^Y_m:=\inf\{k>\zeta^Y_{m-1}:|Y_j|<|Y_k|\leq |Y_l|  \text{ for all } j<k\leq l\}  \]
be regeneration times for the backbone walk. We can then define $\zeta^X_k:=\inf\{m\geq 0: X_m=Y_{\zeta^Y_k}\}$ to be the corresponding regeneration times for $X$. 
By \cite[Proposition 3.4]{lypepe} we have that there exists, $\Pr$-a.s., an infinite sequence of regeneration times $\{\zeta_k^X\}_{k\geq 1}$ and the sequence 
\[\left\{\left(\zeta^X_{k+1}-\zeta^X_k\right),\left(\left|X_{\zeta^X_{k+1}}\right|-\left|X_{\zeta^X_{k}}\right|\right)\right\}_{k\geq1}\]
is i.i.d.\ (as is the corresponding sequence for $Y$). Furthermore, letting $m_t:=\sup\{j\geq 0:\zeta_j^X \leq t\}$ be the number of regenerations by time $t$, we have that $m_t$ is non-decreasing and diverges $\Pb$-a.s.
 
By \cite[Theorems 3.1 \& 4.1]{lypepe}, whenever $\mu>1$ and $\mu^{-1}<\beta<f'(q)^{-1}$ we have that there exists $\nu\in(0,1)$ such that $|X_n|n^{-1}$ converges $\Pb$-a.s.\ to $\nu$. Moreover, combined with \cite[Corollary 3.5]{lypepe}, we have that the time and distance between regenerations of $X$ both have finite means with respect to $\Pb$. Let
\[\chi_j\;:=\;X_{\zeta_j^X}-X_{\zeta_{j-1}^X}-\nu(\zeta_j^X-\zeta_{j-1}^X)\;=\;Y_{\zeta_j^Y}-Y_{\zeta_{j-1}^Y}-\nu(\zeta_j^X-\zeta_{j-1}^X).\]
By the previous remark we have that $\chi_j$ are i.i.d.\ with respect to $\Pb$. By the strong law of large numbers and the definition of $\nu$ we have that $\chi_j$ are centred (see \cite[Theorems 3.1 \& 4.1]{lypepe}). We will show that $\chi_j$ have finite second moment and that their sum
\[\Sigma_m\;:=\;\sum_{j=2}^m\chi_j\;=\;\left(X_{\zeta_m^X}-\nu\zeta_m^X\right)-\left(X_{\zeta_1^X}-\nu\zeta_1^X\right)\]
can be used to approximate $B_t^n$. 

By the remark preceding Lemma \ref{l:subComp}, the offspring distribution $\xi^g$ of $\Tc^g$ has exponential moments. Since $Y$ is a random walk on $\Tc^g$, by \cite[Proposition 3]{peze} we have that $\Eb[(\zeta^Y_2-\zeta^Y_1)^k]<\infty$ for all $k \in \Zb$ whenever $\beta>\mu^{-1}$. 

Let $\eta_k:=r(k+1)-r(k)$ denote the total time taken between $X$ making the $k\th$ and $(k+1)\th$ transition along the backbone. This time consists of
\[N_k:=\sum_{j=r(k)+1}^{r(k+1)}\ind_{\{X_j=Y_k\}}\]
excursions into the finite trees appended to the backbone at this vertex and one additional step to the next backbone vertex. Therefore, we can write
\begin{flalign}\label{e:etaK}\eta_k:=1+\sum_{j=1}^{N_k}\gamma_{k,j} \qquad \text{ where } \qquad \gamma_{k,j}:=\sum_{i=r(k)}^{r(k+1)}\ind_{\{\sum_{l=r(k)}^{i}\ind_{\{X_l=Y_k\}}=j\}}\end{flalign}
is the duration of the $j\th$ such excursion.  

\begin{prp}\label{p:SupSecMom}
Under the assumptions of Theorem \ref{t:SupQCLT} we have that \[\Eb[(\zeta^X_2-\zeta^X_1)^2]<\infty.\]
 \begin{proof}
The law of $\zeta^X_2-\zeta^X_1$ under $\Pb$ is identical to its law under $\Pb_\rho(\cdot|\zeta^Y_1=1)$. That is, by the independence structure, we can condition the first regeneration vertex to be the first vertex reached by $Y$ without changing the law of $\zeta^X_2-\zeta^X_1$. We therefore have that $\Eb\left[(\zeta^X_2-\zeta^X_1)^2\right] $ can be written as
\begin{flalign*}
\Eb\left[(\zeta^X_2-\zeta^X_1)^2\big|\zeta^Y_1=1\right] 
\; = \; \Eb\left[\left(\sum_{k=1}^{\zeta^Y_2-\zeta^Y_1}\eta_k\right)^2\Big{|}\zeta^Y_1=1\right] 
\; \leq \; \Eb\left[(\zeta^Y_2-\zeta^Y_1)\sum_{k=1}^{\zeta^Y_2-\zeta^Y_1}\eta_k^2\Big{|}\zeta^Y_1=1\right] 
\end{flalign*}
by convexity. Using convexity again with the decomposition (\ref{e:etaK}) we can write this as 
\begin{flalign*}
  \Eb\left[(\zeta^Y_2-\zeta^Y_1)\sum_{k=1}^{\zeta^Y_2-\zeta^Y_1} \!\left(1+\sum_{j=1}^{N_k}\gamma_{k,j}\right)^2\!\Big{|}\zeta^Y_1=1\right] 
  \leq \Eb\left[(\zeta^Y_2-\zeta^Y_1)\sum_{k=1}^{\zeta^Y_2-\zeta^Y_1}(N_k+1)\left(1+\sum_{j=1}^{N_k}\gamma_{k,j}^2\right) \Big{|}\zeta^Y_1=1\right]. 
\end{flalign*}

By conditioning on the backbone, buds and the walk on the backbone and buds we have that the individual excursion times are independent of the regeneration times of $Y$ and the number of excursions. The excursion times are also distributed as the first return time to $\oR$ for a walk started from $\oR$ on $\oT^h$. We therefore have that the above expectation can be bounded above by 
\begin{flalign*}
 \Er\left[\Et^{\oT^h}_{\oR}\left[(\tau^+_{\oR})^2\right]\right] \Eb\left[(\zeta^Y_2-\zeta^Y_1)\sum_{k=1}^{\zeta^Y_2-\zeta^Y_1}(N_k+1)^2 \Big{|}\zeta^Y_1=1\right].
\end{flalign*}
Where, by Lemma \ref{l:subComp}, we have that $\Er\left[\Et^{\oT^h}_{\oR}\left[(\tau^+_{\oR})^2\right]\right] <\infty$.
% \begin{flalign}
% \Er\left[\Et^{\oT^h}_{\oR}\left[(\tau^+_{\oR})^2\right]\right] <\infty.
% \end{flalign}

Let $(z_j)_{j=0}^\infty$ denote the ordered distinct vertices visited by $Y$ and 
\[L(z,j):=\sum_{i=0}^j \ind_{\{Y_j=z\}}, \quad L(z):=L(z,\infty)\]
the local times of the vertex $z$. Write 
\[W_{z,l}:=\sum_{j=0}^\infty \ind_{\left\{X_j=z,\; X_{j+1} \notin \Tc^g, \; L(z,j)=l\right\}}\]
to be the number of excursions from $z$ (by $X$) on the $l^{th}$ visit to $z$ (by $Y$) for $l=1,...,L(z)$ and $M:=|\{Y_k\}_{k=1}^{\zeta^Y_2-1}|$ the number of distinct vertices visited by $Y$ between time $1$ and time $\zeta^Y_2-1$ then 
\begin{flalign}
& \Eb\left[(\zeta^Y_2-\zeta^Y_1)\sum_{k=1}^{\zeta^Y_2-\zeta^Y_1}(N_k+1)^2 \Big{|}\zeta^Y_1=1\right]  \notag \\
& \qqquad  = \Eb\left[(\zeta^Y_2-\zeta^Y_1)\sum_{k=1}^{M}\sum_{l=1}^{L(z_k)}(W_{z_k,l}+1)^2 \Big{|}\zeta^Y_1=1\right] \notag \\
 & \qqquad = \sum_{k=1}^{\infty}\sum_{l=1}^{\infty}\Eb\left[(\zeta^Y_2-\zeta^Y_1)\ind_{\{k\leq M, \; l\leq L(z_k)\}}(W_{z_k,l}+1)^2 |\zeta^Y_1=1\right] \notag \\
 & \qqquad \leq \sum_{k=1}^{\infty}\sum_{l=1}^{\infty}\Big(\Eb\left[(\zeta^Y_2-\zeta^Y_1)^2\ind_{\{k\leq M, \; l\leq L(z_k)\}}|\zeta^Y_1=1\right]\Eb\left[(W_{z_k,l}+1)^4 |\zeta^Y_1=1\right]\Big)^{1/2} \label{e:CSBnd}
\end{flalign}
by Cauchy-Schwarz. Conditional on $\zeta^Y_1=1$, for all $1\leq k\leq M$ we have that $L(z_k)\leq \zeta^Y_2-\zeta^Y_1$; moreover, $M\leq \zeta^Y_2-\zeta^Y_1$ therefore \[\ind_{\{k\leq M, \; l\leq L(z_k)\}}\leq\ind_{\{k,l\leq \zeta^Y_2-\zeta^Y_1\}}.\] 

Since the root does not have a parent, without any further information concerning the number of children from a given vertex, we have that the walk is more likely to take an excursion into one of the neighbouring traps when at the root than from this vertex. We can, therefore, stochastically dominate the number of excursions from a vertex by the number of excursions from the root to see that $\Eb\left[(W_{z_k,l}+1)^4 \right]\leq \Eb\left[(W_{z_0,1}+1)^4 \right]$. Using this, Cauchy-Schwarz and that $\Pb(\zeta^Y_1=1)>0$, the expression (\ref{e:CSBnd}) can be bounded above by
\begin{flalign*}
  & \Pb(\zeta^Y_1=1)^{-1}\sum_{k=1}^{\infty}\sum_{l=1}^{\infty}\left(\Eb\left[(\zeta^Y_2-\zeta^Y_1)^2\ind_{\{k,l\leq \zeta^Y_2-\zeta^Y_1\}}\right]\Eb\left[(W_{z_k,l}+1)^4 \right]\right)^{1/2} \\
  & \qqqquad \qqqquad \leq C\Eb\left[(\zeta^Y_2-\zeta^Y_1)^4\right]^{1/4}\Eb\left[(W_{z_0,1}+1)^4 \right]^{1/2}\sum_{k=1}^{\infty}\sum_{l=1}^{\infty}\Pb\left(k,l\leq \zeta^Y_2-\zeta^Y_1\right)^{1/4}.
\end{flalign*}
% We therefore have that this is bounded above by
% \begin{flalign*}
%  C\Eb\left[(W_{z_0,1}+1)^4 \right]\Eb\left[(\zeta^Y_2-\zeta^Y_1)^2\sum_{k=1}^{\infty}\sum_{l=1}^{\infty}\ind_{\{k,l\leq \zeta^Y_2-\zeta^Y_1\}}\right]\;=\;C\Eb\left[(W_{z_0,1}+1)^4 \right]\Eb\left[(\zeta^Y_2-\zeta^Y_1)^4\right].
% \end{flalign*}
Since the offspring distribution $\xi^g$ has exponential moments we have that the time between regenerations has finite fourth moments by \cite[Proposition 3]{peze}. That is, $\Eb\left[(\zeta^Y_2-\zeta^Y_1)^4\right]<\infty$.

Write $Z_n$ and $Z^g_n$ to be the GW-processes associated with $\Tc$ and $\Tc^g$. The number of excursions from the root is geometrically distributed with termination probability $1-p_{ex}$ where 
\[p_{ex}:=\frac{Z_1-Z_1^g}{Z_1}.\]
Using properties of geometric random variables we therefore have that  
\[\Eb\left[(W_{z_0,1}+1)^4 \right] \;\leq \; C\Eb[(1-p_{ex})^{-4}] \;\leq \; C\Eb[Z_1^4]  \;< \;\infty\]
since $Z_1\ed \xi$ which has exponential moments. 

It remains to show that 
\begin{flalign}\label{e:dubsum}
\sum_{k=1}^{\infty}\sum_{l=1}^{\infty}\Pb\left(k,l\leq \zeta^Y_2-\zeta^Y_1\right)^{1/4}
\end{flalign}
is finite. Note that $\Pb\left(k,l\leq \zeta^Y_2-\zeta^Y_1\right)=\Pb\left(\zeta^Y_2-\zeta^Y_1\geq l\right)$ whenever $l\geq k$. Using Chebyshev's inequality we can then bound (\ref{e:dubsum}) above by
\begin{flalign*}
 2\sum_{k=1}^{\infty}\sum_{l=k}^{\infty}\Pb\left(\zeta^Y_2-\zeta^Y_1\geq l\right)^{1/4} \leq 2\sum_{k=1}^{\infty}\sum_{l=k}^{\infty} \left(\frac{\Eb\left[\left(\zeta^Y_2-\zeta^Y_1\right)^j\right]}{l^j}\right)^{1/4}
\end{flalign*}
for any integer $j$. In particular, by \cite[Proposition 3]{peze} we have that $\Eb\left[\left(\zeta^Y_2-\zeta^Y_1\right)^j\right]$ is finite for any integer $j$. Choosing $j>8$ we then have that this sum is finite which completes the proof.
\end{proof}
\end{prp}

For $x \in \Tc$ let $\Tc_x$ denote the subtree consisting of all descendants of $x$ in $\Tc$. Then, for $y \in \Tc^g$, let $\Tc_y^{-}$ be the branch at $y$; that is, the subtree rooted at $y$ consisting only of $y$, the children of $y$ not on $\Tc^g$ and their descendants. The tree $\Tc_y^{-}$ then has the law of a tree rooted at $y$ with some random number $M_y^-$ of $h$-GW-trees attached to $y$. Since $M_y^-$ is dominated by $\xi$, by (\ref{e:expmom}) we have that $M_y^-$ has exponential moments. It therefore follows from \cite[Theorem B]{lypepe} that there exists some constant $C$ such that
\begin{flalign}\label{e:hgtbrn}
\Pr(\Hc(\Tc_y^-)\geq n)\leq Cf'(q)^{n}
\end{flalign}
where, for a fixed rooted tree $\Ts$, $\Hc(\Ts):=\sup\{d(\rho,x):x \in \Ts\}$ is the height of $\Ts$. Let $\Hc_n:=\max\{\Hc(\Tc_y^{-}):y \in \{Y_k\}_{k=0}^n\}$ denote the largest branch seen by $Y$ by time $n$. It follows that
\begin{flalign}\label{e:BSig}
\sup_{t \in [0,T]}\left|B_t^n-\frac{\Sigma_{m_{tn}}}{\sigma\sqrt{n}}\right|\leq 
\frac{|X_{\zeta_1^X}|+\nu\zeta_1^X+\Hc_{nT}}{\sigma\sqrt{n}}+\sup_{j=1,...,m_{nT}}\frac{|Y_{\zeta^Y_{j+1}}|-|Y_{\zeta^Y_j}|+\nu(\zeta^X_{j+1}-\zeta^X_j)}{\sigma\sqrt{n}}.
\end{flalign}

Up to time $nT$, the walk $Y$ can have visited at most $nT$ vertices on $\Tc^g$ therefore the probability that $X$ has visited a branch of height at least $C\log(n)$ by time $nT$ is at most $C_Tnf'(q)^{C\log(n)}$. In particular, by Borel-Cantelli, choosing $C$ suitably large we have that there are almost surely only finitely many $n$ such that $Y$ has visited the root of a branch of height at least $C\log(n)$ by time $nT$. Since $X_{\zeta_1^X}$ and $\zeta_1^X$ do not depend on $n$ and have finite mean, we have that the first term in (\ref{e:BSig}) converges $\Pb$-a.s.\ to $0$.

By \cite[Proposition 3]{peze}, for any $k \in \Zb^+$ we have that $\Eb[(|Y_{\zeta^Y_1}|-|Y_{\zeta^Y_1}|)^k]<\infty$ therefore the distance between regeneration points is small. In particular, bounding $m_{nT}$ above by $nT$, using a union bound and Markov's inequality we have that for any $\varepsilon>0$,
\[\Pb\left(\sup_{j=1,...,m_{nT}}\frac{|Y_{\zeta^Y_{j+1}}|-|Y_{\zeta^Y_j}|}{\sigma\sqrt{n}}>\varepsilon\right)\leq C_{T,\varepsilon}\Eb\left[\left(|Y_{\zeta^Y_2}|-|Y_{\zeta^Y_1}|\right)^2\ind_{\left\{|Y_{\zeta^Y_2}|-|Y_{\zeta^Y_1}|>\varepsilon\sqrt{n}\right\}}\right]\]
which converges to $0$ as $\nin$ by dominated convergence. Similarly, using Proposition \ref{p:SupSecMom}, we have that the same holds for the supremum of $\zeta^X_{j+1}-\zeta^X_j$; therefore, we have that
\[\Pb\left(\sup_{t \in [0,T]}\left|B_t^n-\frac{\Sigma_{m_{tn}}}{\sigma\sqrt{n}}\right|>\varepsilon\right)\]
converges to $0$ as $\nin$.

By the law of large numbers and that $\zeta^X_1/n$ converges $\Pb$-a.s.\ to $0$ we have that
\[\zeta^X_n=\zeta^X_1+\sum_{k=2}^n(\zeta^X_k-\zeta^X_{k-1})\]
converges $\Pb$-a.s. It therefore follows by continuity of the inverse at strictly increasing functions, \cite[Corollary 13.6.4]{wh}, that $m_{nt}/n$ converges $\Pb$-a.s.\ to a deterministic linear process. 

By Proposition \ref{p:SupSecMom} and the remark leading to it we have that $\Sigma_m$ is the sum of i.i.d.\ centred random variables with finite second moment. By Donsker's theorem we therefore have that $(\Sigma_{nt}/\sqrt{n})_{t\geq 0}$ converges to a scaled Brownian motion. By continuity of composition at continuous limits, \cite[Theorem 13.2.1]{wh}, and the previous remarks we therefore have the following annealed central limit theorem.
\begin{cly}\label{c:AnnSupFCLT}
Under the assumptions of Theorem \ref{t:SupQCLT}, there exists a constant $\sigma^2>0$ such that the process \[B_t^n:=\frac{|X_{\lfloor nt\rfloor}|-n \nu t}{\sigma\sqrt{n}}\] converges in $\Pb$-distribution on $D([0,\infty),\Rb)$ endowed with the Skorohod $J_1$ topology to a standard Brownian motion.
\end{cly}

\begin{rem}\label{r:BetaMu}
The branch of a subcritical GW-tree conditioned to survive can be constructed by attaching a random number of subcritical GW-trees to a root vertex. In \cite[Lemma 4.12]{bowdpre} it is shown that, conditional on having a single vertex in the first generation of the branch, the second moment of the first return time to the root is infinite whenever $\beta^2\tilde{\mu}\geq 1$ where $\tilde{\mu}$ is the mean of the subcritical GW-law. It therefore follows from this that 
\begin{flalign*}
\Er\left[\Et^{\oT^h}_{\oR}\left[(\tau^+_{\oR})^2\right]\right] =\infty
\end{flalign*}
whenever $\beta^2f'(q)\geq1$ and $\mu>1$. In particular, if we have that $\beta^2f'(q)\geq 1$ then $\chi_j$ have infinite second moments since Proposition \ref{p:SupSecMom} fails. In this case, we do not have a central limit theorem for $\Sigma_m$ from which it follows that $B_t^n$ does not converge in distribution. This shows that the condition $\beta^2f'(q)<1$ is necessary for the annealed central limit theorem. We note here that when $p_0=0$ we have that $q=0=f'(q)$ and, therefore, this condition is necessarily satisfied.
\end{rem}

\section{A quenched invariance principle}\label{s:SQue}
We now extend Corollary \ref{c:AnnSupFCLT} to a quenched functional central limit theorem. For each $n \in \Nb$ write $\Bb_t^n(X)$ to be the linear interpolation satisfying 
\[\Bb_{k/n}^n(X)=\frac{|X_k|-k\nu}{\sigma\sqrt{n}}\] 
for $k \in \Nb$. We then have that $B_t^n=\Bb_t^n$ for $t>0$ such that $nt \in \Nb$ and $|B_t^n-\Bb_t^n|\leq n^{-1/2}(\nu+1)/\sigma$ therefore it suffices to consider the interpolation. To begin, we prove the following lemma which is the analogue of \cite[Lemma 4.1]{bosz} and follows by the same method. 

\begin{lem}\label{l:BoSz}
 Suppose that the assumptions of Theorem \ref{t:SupQCLT} hold and that for any bounded Lipschitz function $F:C([0,T],\Rb)\rightarrow \Rb$ and $b\in(1,2)$ we have that 
 \begin{flalign}\label{e:BoSz}
 \sum_{k\geq 1}\mathrm{Var}_\Pr\left(\Et^\Tc\left[F\left(\Bb^{\lfloor b^k\rfloor}\right)\right]\right)<\infty.
\end{flalign}
Then, for $\Pr$-a.e.\ $\Tc$, the process $(B_t^n)_{t\geq 0}$ converges in $\Pt^\Tc$-distribution on $D([0,\infty),\Rb)$ endowed with the Skorohod $J_1$ topology to a standard Brownian motion.
 \begin{proof}
Suppose that for any bounded Lipschitz function $F:C([0,T],\Rb)\rightarrow \Rb$ and $b\in(1,2)$ we have that $\Pr$-a.s.\ 
\begin{flalign}\label{e:QFconv}
\Et^\Tc[F(\Bb^{\lfloor b^k\rfloor})] \rightarrow \Et[F(B)]
\end{flalign}
where $B$ is a standard Brownian motion. For any $\delta,T>0$, the function $F_{T,\delta}(\omega):=\sup\{|\omega(s)-\omega(t)|\land 1:s,t\leq T,\; |t-s|\leq \delta\}$ is bounded and Lipschitz; furthermore, for $\Pr$-a.e.\ $\Tc$ 
\begin{flalign}\label{e:QFT}
\lim_{\delta \rightarrow 0}\limsup_{k\rightarrow \infty} \Et^\Tc[F_{T,\delta}(\Bb^{\lfloor b^k\rfloor})]=0
\end{flalign}
since, by properties of Brownian motion, $\Et[F_{T,\delta}(B)] \rightarrow 0$ as $\delta \rightarrow 0$. In particular, by Markov's inequality we then have that for any $\varepsilon>0$
\begin{flalign*}
\lim_{\delta\rightarrow 0}\limsup_{k\rightarrow \infty}\Pt^\Tc\left(\sup_{\substack{s,t\leq T: \\|s-t|\leq \delta  }}|\Bb_s^{\lfloor b^k\rfloor}-\Bb_t^{\lfloor b^k\rfloor}|>\varepsilon\right) \; \leq \; \lim_{\delta\rightarrow 0}\limsup_{k\rightarrow \infty} \varepsilon^{-1}\Et^\Tc\left[F_{T,\delta}(\Bb^{\lfloor b^k\rfloor})\right] \;= \; 0
\end{flalign*}
which gives tightness of $(\Bb^{\lfloor b^k\rfloor}_\cdot)_{k=1}^\infty$ under $\Pt^\Tc$.

For $n \in \Nb$ let $k_n$ denote the unique integer such that $\lfloor b^{k_n} \rfloor\leq n <\lfloor b^{k_n+1}\rfloor$ then, by Markov's inequality and the definition of the interpolation, we have that for any $\varepsilon\in(0,1)$
\begin{flalign*}
 \lim_{\delta\rightarrow 0}\limsup_{k\rightarrow \infty}\Pt^\Tc\left(\sup_{\substack{s,t\leq T: \\|s-t|\leq \delta  }}|\Bb_s^{n}-\Bb_t^{n}|>\varepsilon\right)  
 & \leq \lim_{\delta\rightarrow 0}\limsup_{n\rightarrow \infty} \Et^\Tc\left[\sup_{\substack{s,t\leq T: \\|s-t|\leq \delta  }}|\Bb_s^n-\Bb_t^n|\land 1\right] \varepsilon^{-1} \\
 &  \leq \lim_{\delta\rightarrow 0}\limsup_{n\rightarrow \infty} \Et^\Tc\left[\sup_{\substack{s,t\leq T: \\|s-t|\leq \delta  }}\left|\Bb_{ s\frac{n}{\lfloor b^{k_{\scriptscriptstyle{n}}}\rfloor}}^{\lfloor b^{k_n}\rfloor}-\Bb_{ t\frac{n}{\lfloor b^{k_n}\rfloor}}^{\lfloor b^{k_n}\rfloor}\right|\land 1\right]\varepsilon^{-1} \\
 &  \leq \lim_{\delta\rightarrow 0}\limsup_{k\rightarrow \infty} \Et^\Tc\left[\sup_{\substack{u,v\leq 2T: \\|u-v|\leq 2\delta  }}|\Bb_{u}^{\lfloor b^k\rfloor}-\Bb_{v}^{\lfloor b^k\rfloor}|\land 1\right] \varepsilon^{-1}
\end{flalign*}
since $b<2$ implies that $|s\frac{n}{\lfloor b^{k_n}\rfloor}-t\frac{n}{\lfloor b^{k_n}\rfloor}|<2\delta$ whenever $|s-t|<\delta$. In particular, the above expression is equal to $0$ by (\ref{e:QFT}) therefore the laws of $\Bb^n_\cdot$  are tight under $\Pt^\Tc$.

Let $F$ be bounded and Lipschitz; without loss of generality we may assume $||F||_\infty,||F||_{Lip}\leq 1$. For $n \in \Nb$ we have that 
\begin{flalign*}
 & \limsup_{\nin} |\Et^\Tc[F(\Bb^n)]-\Et^\Tc[F(\Bb^{\lfloor b^{k_n} \rfloor})]| \\
 & \qqquad \leq ||F||_{Lip}\limsup_{n\rightarrow \infty}\Et^\Tc\left[\sup_{s \leq T}|\Bb_s^n-\Bb_s^{\lfloor b^{k_n}\rfloor}|\land 1\right] \\
 & \qqquad \leq \limsup_{n\rightarrow \infty}\Et^\Tc\left[\sup_{s \leq T}\left|\sqrt{\frac{\lfloor b^{k_n}\rfloor}{n}}\Bb_{s\frac{n}{\lfloor b^{k_n}\rfloor}}^{\lfloor b^{k_n}\rfloor}-\Bb_s^{\lfloor b^{k_n}\rfloor}\right|\land 1\right] \\
 & \qqquad \leq \limsup_{n\rightarrow \infty}\Et^\Tc\left[\sup_{s \leq T}\left|\left(\sqrt{\frac{\lfloor b^{k_n}\rfloor}{n}}-1\right)\Bb_{s\frac{n}{\lfloor b^{k_n}\rfloor}}^{\lfloor b^{k_n}\rfloor}\right|\land 1\right]
 +\Et^\Tc\left[\sup_{\substack{s,t\leq T: \\|s-t| \leq T(b-1)}}\left|\Bb_{s}^{\lfloor b^{k_n}\rfloor}-\Bb_t^{\lfloor b^{k_n}\rfloor}\right|\land 1\right] \\
 & \qqquad \leq |1-b^{-1/2}|\Et\left[\sup_{s\leq bT}|B_s|\right] + \limsup_{k\rightarrow \infty}\Et^\Tc\left[F_{T,T(b-1)}(\Bb^{\lfloor b^k\rfloor})\right]
\end{flalign*}
which converges to $0$ as $b\rightarrow 1$. In particular, when (\ref{e:QFconv}) holds for any $F:C([0,T],\Rb)\rightarrow \Rb$ with $||F||_\infty,||F||_{Lip} \leq 1$ and $b\in(1,2)$ $\Pr$-a.s.\ we have that 
\begin{flalign}\label{e:ETE}
\Et^\Tc[F(\Bb^n)] \rightarrow \Et[F(B)].
\end{flalign}
Bounded Lipschitz functions are separable and dense in the space of continuous bounded functions therefore we have that (\ref{e:ETE}) holds $\Pr$-a.s.\ $\forall F \in C_b(C([0,T],\Rb))$ which completes the quenched functional CLT. It therefore remains to show that (\ref{e:BoSz}) implies (\ref{e:QFconv}).

By Corollary \ref{c:AnnSupFCLT} we have that for any bounded Lipschitz function, 
\[\Er\left[\Et^\Tc\left[F(\Bb^{\lfloor b^k\rfloor})\right]\right]=\Eb\left[F(\Bb^{\lfloor b^k\rfloor})\right]\rightarrow \Et[F(B)]\] 
as $k\rightarrow \infty$ therefore, it suffices to show that for $\Pr$-a.e.\ tree $\Tc$ we have that $|\Er[\Et^\Tc[F(\Bb^{\lfloor b^k\rfloor})]]-\Et^\Tc[F(\Bb^{\lfloor b^k\rfloor})]|$ converges to $0$. By Chebyshev's inequality, for $\varepsilon>0$, 
\[\Pr\left(|\Er[\Et^\Tc[F(\Bb^{\lfloor b^k\rfloor})]]-\Et^\Tc[F(\Bb^{\lfloor b^k\rfloor})]|>\varepsilon\right)\leq \varepsilon^{-2}\mathrm{Var}_\Pr\left(\Et^\Tc\left[F\left(\Bb^{\lfloor b^k\rfloor}\right)\right]\right). \]
The result follows from Borel-Cantelli and (\ref{e:BoSz}).
 \end{proof}
\end{lem}

We now complete the proof of the quenched functional CLT by following the method used in \cite{peze} to show that condition (\ref{e:BoSz}) holds for any bounded Lipschitz function $F:C([0,T],\Rb)\rightarrow \Rb$ and $b\in(1,2)$ under the assumptions of the theorem.

\begin{proof}[Proof of Theorem \ref{t:SupQCLT}] 
For a fixed tree $\Tc$, let $X^1,X^2$ be independent $\beta$-biased walks on $\Tc$ and $Y^1,Y^2$ the corresponding backbone walks. For $i=1,2, \; k \in \Nb$ and $t,s\geq 0$ let 
\[\Bb_{t,s}^{k,i}=\Bb^{\lfloor b^k\rfloor}_t(X^i_{\cdot+s})-\Bb^{\lfloor b^k\rfloor}_t(X^i_s)\]
be a random variable with law of the interpolation $\Bb^{\lfloor b^k\rfloor}$ started from the vertex $X_s^i$. Define 
\[\vartheta_k^{Y^i}:=\min\{m>\lfloor b^{k/4}\rfloor: m\in\{\zeta_j^{Y^i}\}_{j\geq 1}\} \qquad \text{and} \qquad \vartheta_k^{X^i}=\min\left\{m\geq 0:X_m^i=Y^i_{\vartheta_k^{Y^i}}\right\}\]
to be the first regeneration time of $Y^i$ after time $\lfloor b^{k/4}\rfloor$ and the corresponding time for $X^i$. 

Let 
\begin{flalign*}
\Ac_k^1 & :=\left\{\{Y^1_s: \; s\leq \vartheta_k^{Y^1}\} \cap \{Y^2_{\vartheta_k^{Y^2}}\}=\phi\right\}=\left\{\{X^1_s: \;s\leq \vartheta_k^{X^1}\} \cap \{X^2_{\vartheta_k^{X^2}}\}=\phi\right\},\\
\Ac_k^2 & :=\left\{\{Y^2_s: \;s\leq \vartheta_k^{Y^2}\} \cap \{Y^1_{\vartheta_k^{Y^1}}\}=\phi\right\}=\left\{\{X^2_s: \;s\leq \vartheta_k^{X^2}\} \cap \{X^1_{\vartheta_k^{X^1}}\}=\phi\right\}
\end{flalign*}
and $\Ac_k:=\Ac_k^1\cap \Ac_k^2$ be the event that, after the first regeneration times after time $\lfloor b^{k/4}\rfloor$, the paths of $Y^1,Y^2$ do not intersect. Write $\Bc^{k,i}:=\{\vartheta_k^{Y^i}\leq b^{k/3}\}$ to be the event that the first regeneration after time $b^{k/4}$ happens before time $b^{k/3}$.

Recall that for $x \in \Tc^g$ we denote by $\Hc(\Tc^-_x)$ the height of the branch attached to the vertex $x$. Using Lipschitz properties of $\Bb^{k,i}$ we have that
\begin{flalign*}
 \sup_{t\leq T}\left|\Bb^{k,i}_{t,0}-\Bb^{k,i}_{t,\vartheta_k^{X^i}}\right| & \leq \sup_{m\leq Tb^k}b^{-k/2}\left||X^i_m|-m\nu-|X^i_{m+\vartheta_k^{X^i}}|+(m+\vartheta_k^{X^i})\nu+|X^i_{\vartheta_k^{X^i}}|-\vartheta_k^{X^i}\nu\right| \\
 & =\sup_{m\leq Tb^k}b^{-k/2}\left||X^i_m|-|X^i_{m+\vartheta_k^{X^i}}|+|X^i_{\vartheta_k^{X^i}}|\right| \\
 & \leq b^{-k/2}\max_{m\leq Tb^k}\left|\left||Y^i_m|-|Y^i_{m+\vartheta_k^{Y^i}}|\right|+|Y^i_{\vartheta_k^{Y^i}}|\right| + b^{-k/2}\Hc^i_{Tb^k}
\end{flalign*}
where $\Hc^i_{Tb^k}$ is the height of the tallest branch seen by time $Tb^k$ by $Y^i$. By time $Tb^k$ the walk $Y^i$ can visit at most $Tb^k+1$ unique vertices. At the first hitting time of a vertex, the bud and backbone distribution from this vertex are independent of the past; therefore, by (\ref{e:hgtbrn})
\begin{flalign}
 \Pb\left(\Hc^i_{Tb^k}\geq C\log(b^k)\right) 
 \;\leq\; C_Tb^k \Pr(\Hc(\Tc_\rho^-)\geq C\log(b^k))  
 \;\leq\; C_Tb^kf'(q)^{C\log(b^k)} 
 \;\leq\; C_Tb^{-k} \label{e:shTr}
\end{flalign}
for $C$ sufficiently large. Furthermore, by the Lipschitz property of $Y^i$ we have that
\begin{flalign*}
 b^{-k/2}\max_{m\leq Tb^k}\left|\left||Y^i_m|-|Y^i_{m+\vartheta_k^{Y^i}}|\right|+|Y^i_{\vartheta_k^{Y^i}}|\right| & \leq 2\vartheta_k^{Y^i}b^{-k/2}
\end{flalign*}
which is bounded above by $2b^{-k/6}$ on the event $\Bc^{k,i}$. Letting $\Cc^{k,i}:=\{\Hc^i_{Tb^k}< C\log(b^k)\}$, we then have that, on the event $\Bc^{k,i} \cap \Cc^{k,i}$,
\[\left|F\left(\Bb^{k,i}_{\cdot,0}\right)-F\left(\Bb^{k,i}_{\cdot,\vartheta_k^{X^i}}\right)\right|\leq Cb^{-k/6}\]
for any bounded Lipschitz function $F:C([0,T],\Rb)\rightarrow \Rb$.

Using the Lipschitz and boundedness properties of $F$, we then have that
\begin{flalign*}
 & \mathrm{Var}_\Pr\left(\Et^\Tc\left[F\left(\Bb^{\lfloor b^k\rfloor}\right)\right]\right) \\
 &  = \Er\left[\Et^\Tc\left[F(\Bb^{\lfloor b^k\rfloor})\right]^2\right]-\Er\left[\Et^\Tc\left[F(\Bb^{\lfloor b^k\rfloor})\right]\right]^2 \\
 &  = \Eb\left[F(\Bb^{k,1})F(\Bb^{k,2})\right]-\Eb\left[F(\Bb^{k,1})\right]\Eb\left[F(\Bb^{k,2})\right] \\
 &  \leq C\left(\Pb\left((\Bc^{k,1})^c\right)\!+\Pb\left((\Cc^{k,1})^c\right)\!+b^{-k/6}\right)
 \!+\Eb\left[F\!\left(\Bb^{k,1}_{\cdot,\vartheta_k^{X^1}}\!\right)F\!\left(\Bb^{k,2}_{\cdot,\vartheta_k^{X^2}}\!\right)\right]
 \!-\Eb\left[F\!\left(\Bb^{k,1}_{\cdot,\vartheta_k^{X^1}}\!\right)\right]\Eb\left[F\!\left(\Bb^{k,2}_{\cdot,\vartheta_k^{X^2}}\!\right)\right].
\end{flalign*}

On the event $\Ac_k$ we have that $\Bb^{k,1}_{\cdot,\vartheta_k^{X^1}}, \; \Bb^{k,2}_{\cdot,\vartheta_k^{X^2}}$ are independent therefore \[\Eb\left[F\left(\Bb^{k,1}_{\cdot,\vartheta_k^{X^1}}\right)F\left(\Bb^{k,2}_{\cdot,\vartheta_k^{X^2}}\right)|\Ac_k\right]-\Eb\left[F\left(\Bb^{k,1}_{\cdot,\vartheta_k^{X^1}}\right)|\Ac_k\right]\Eb\left[F\left(\Bb^{k,2}_{\cdot,\vartheta_k^{X^2}}\right)|\Ac_k\right]=0.\]
Using the Lipschitz property of $F$ we then have that 
\begin{flalign*}
 \mathrm{Var}_\Pr\left(\Et^\Tc\left[F\left(\Bb^{\lfloor b^k\rfloor}\right)\right]\right) \leq C\left(\Pb\left((\Ac^{k,1})^c\right)+\Pb\left((\Bc^{k,1})^c\right)+\Pb\left((\Cc^{k,1})^c\right)+b^{-k/6}\right).
\end{flalign*}

For $i=1,2$ we have that $Y^i$ are biased random walks on a supercritical GW-tree without leaves $\Tc^g$, whose offspring law has exponential moments. It follows that the estimates $\Pb((\Ac^{k,1})^c),\Pb((\Bc^{k,1})^c)\leq b^{-\tilde{c}k}$ given in the proof of \cite[Theorem 3]{peze} still hold. Combining these with (\ref{e:shTr}) we have that there exists $c>0$ such that for $k$ sufficiently large 
\[\mathrm{Var}_\Pr\left(\Et^\Tc\left[F\left(\Bb^{\lfloor b^k\rfloor}\right)\right]\right) \leq Cb^{-ck}\]
which shows (\ref{e:BoSz}) and therefore the result follows from Lemma \ref{l:BoSz}.

 \end{proof}

\section*{Acknowledgements}
I would like to thank my supervisor David Croydon for suggesting the problem, his support and many useful discussions. This work is supported by EPSRC as part of the MASDOC DTC at the University of Warwick. Grant No.\ EP/H023364/1.

\end{document}